\documentclass[a4paper,12pt]{article}
\usepackage[utf8]{inputenc}
\usepackage{amssymb}
\usepackage{amsmath}
\usepackage{amsthm}
\usepackage{enumitem}
\usepackage{color}
\usepackage{soul}
\usepackage{hyperref}

\newtheorem{thm}{Theorem}

\newtheorem{lem}[thm]{Lemma}
\newtheorem{cor}[thm]{Corollary} 
\newtheorem{prop}[thm]{Proposition}
 
\newtheorem*{op*}{Open Problem}

\theoremstyle{definition}  
\newtheorem{remark}{Remark}

\newcommand{\N}{\mathbb{N}}

\newcommand{\IN}{\mathbb{N}}

\newcommand{\IR}{\mathbb{R}}

\newcommand{\vol}{\mathrm{vol}}
\newcommand{\dd}{\mathrm{d}}

\DeclareMathOperator{\id}{id}

\title{Entropy numbers of finite-dimensional\\ Lorentz space embeddings}
\author{Joscha Prochno\footnote{Faculty of Computer Science and Mathematics, University of Passau, Dr.-Hans-Kapfinger-Str. 30, 94032 Passau, Germany, {\tt joscha.prochno@uni-passau.de}}, 
Mathias Sonnleitner$^{*,}$\footnote{Institute of Mathematical Stochastics, University of M\"unster, Orl\'eans-Ring 10, 48149 M\"unster, Germany, {\tt mathias.sonnleitner@uni-muenster.de}} 
and Jan Vyb\'iral\footnote{Department of Mathematics, Faculty of Nuclear Sciences and Physical Engineering,
Czech Technical University, Trojanova 13, 12000 Praha, Czech Republic, {\tt jan.vybiral@fjfi.cvut.cz}}
}

\begin{document} 

\maketitle

\begin{center}
\emph{Dedicated to the memory of Albrecht Pietsch.}
\vskip 3mm
\end{center}

\abstract{
\noindent The sequence of entropy numbers quantifies the degree of compactness of a linear operator acting between quasi-Banach spaces. We determine the asymptotic behavior of entropy numbers in the case of natural embeddings
between finite-dimensional Lorentz spaces $\ell_{p,q}^n$ in all regimes; our results are sharp up to constants. This generalizes classical results obtained by Schütt (in the case of Banach spaces) and Edmunds and Triebel, K\"uhn,
as well as Gu\'edon and Litvak (in the case of quasi-Banach spaces) for entropy numbers of identities between finte-dimensional Lebesgue sequence spaces $\ell_p^n$. We employ techniques such as interpolation, volume comparison as well as techniques from sparse
approximation and combinatorial arguments.  Further, we characterize entropy numbers of embeddings between finite-dimensional symmetric quasi-Banach spaces in terms of best $s$-term approximation numbers. 
}

\section{Introduction and main result}

The fundamental notion of covering numbers and with it Pietsch's inverse concept of entropy numbers \cite{P1978_operator_ideals} quantify to what extent a bounded linear operator is compact.
They are important elements in both pure and applied mathematics, for instance, in the geometry of Banach spaces \cite{BLM1989,C1981,GKS1987,K1986},
signal processing and compressed sensing \cite{CRT2006,Don06,FR13}, or the theory of random processes \cite{LT1991,LL2002,T2005}. In particular, they are useful measures for complexity in approximation theory \cite{CS90,ET96} and a powerful tool in the flourishing field of statistical machine learning, explaining, for instance, the effect of the choice of kernel function on the generalization performance of support vector machines \cite{WSS2001} (see also \cite{CS02}). The origin of these notions can in fact be traced back to Kolmogorov, who, motivated by ideas and definitions in information theory, introduced the so-called $\varepsilon$-entropy already in the $1950$s \cite{K1956}.  

From the point of view of geometric functional analysis, entropy numbers are quite well understood in a number of fundamental and important situations, but before presenting explicit examples, let us further motivate the interest
in them by looking at their relation to a specific sequence of singular numbers ($s$-numbers) of operators and the problem of optimal recovery.  Via a famous inequality of Carl \cite{C1981}, entropy numbers are related to the most important scales of $s$-numbers and, specifically, can provide lower bounds on the so-called Gelfand numbers,
which bound from below the error of optimal reconstructions using linear measurements (see, e.g., \cite{NW08} for background information on the related field of information-based complexity); one should note that,
in general, it is significantly more delicate to determine the asymptotic behavior of Gelfand numbers than of entropy numbers. Let $n\in\N$ and consider continuous linear functionals $L_1,\dots,L_n$ on a quasi-normed space $F$.
The problem is to approximate in the quasi-norm of a quasi-normed space $G$ into which $F$ continuously embeds an unknown element $f$ from the unit ball of $F$ purely based on the measurements $L_1(f),\dots,L_n(f)$.
Then the worst-case error of any such approximation $A(f)=\varphi(L_1(f),\dots,L_n(f))$, where $\varphi\colon \IR^n\to G$ is an algorithm using the linear information, is bounded from below by the $n$-th Gelfand number
of the natural identity $\id\colon F\to G$ (see, e.g., \cite[Proposition 1.2]{FPR+10}). In the context of compressed sensing, where one is interested in recovery of (nearly) sparse signals, this has been used by Donoho~\cite{Don06} in the case of the embedding $\id\colon \ell_p^n \to \ell_2^n$ with the claimed extension of Carl's inequality to quasi-Banach spaces only proven later by Hinrichs, Kolleck, and Vyb\'iral \cite{HKV16}.
The lower bound for the associated Gelfand numbers has been proven earlier by Foucart, Pajor, Rauhut, and Ullrich \cite{FPR+10}.   

Of specific interest, as indicated above, are often finite-dimensional embeddings, also because they can serve as a discrete model for operators between function spaces, such as differential operators between Sobolev spaces \cite{ET96,K1986}.
Arguably most fundamental in this respect are the natural embeddings of Lebesgue sequence spaces, i.e., of  $\id:\ell_p^n\to\ell_q^n$, and in this situation the behavior of entropy numbers is in fact well understood. Indeed, in the case of Banach spaces, this is a classical result of Schütt~\cite{Sch84} (who actually obtained more general results for entropy numbers of diagonal operators between symmetric Banach sequence spaces), while its extension to the quasi-Banach space setting has been obtained by Edmunds and Triebel \cite[Sec. 3.2.2]{ET96}, Kühn \cite{Kue01}, and independently by Gu\'{e}don and Litvak \cite{GL00}; we refer to the survey \cite{KV20} by Kossaczk\'a and Vyb\'iral (see also \cite[Remark 3]{MU21}) for an account on the history of this result.

Before we state the result, let us recall that the $k$-th (dyadic) entropy number of a continuous linear map $T\colon X\to Y$ between quasi-Banach spaces with unit balls $B_X$ and $B_Y$, respectively, is given by
\[
e_k(T\colon X\to Y):=\inf\Bigg\{\varepsilon>0 \colon \exists\, y_1,\dots,y_{2^{k-1}}\in Y \colon T(B_X)\subset \bigcup_{i=1}^{2^{k-1}} (y_i+\varepsilon B_Y)\Bigg\}.
\]
Moreover, entropy numbers are almost $s$-numbers and satisfy
\begin{enumerate}
	\item (norming property) $2^{1-1/p}\|T\|\le e_1(T)\le \|T\|$, whenever $Y$ is a $p$-Banach space,
	\item (monotonicity) $e_1(T)\ge e_2(T)\ge \cdots \ge 0$,
	\item (sub-multiplicativity) $e_{n+m-1}(ST)\le e_n(S)\cdot e_m(T)$ for $n,m\in\IN$ and $S\colon Y\to Z$ is a linear and continuous map to a quasi-Banach space $Z$.
\end{enumerate}
These properties can be deduced, e.g.,  from \cite[Lemma 1.3.1.1]{ET96} and its proof; see also Section~\ref{sec:volume} for more information on quasi- and $p$-Banach spaces. 

 For a sequence $x=(x_i)_{i\in\IN}\in\IR^{\IN}$, we denote
\[
\|x\|_p:=
\begin{cases}
	\displaystyle\bigg(\sum_{i=1}^{\infty}|x_i|^p\bigg)^{1/p}&\colon 0<p<\infty,\\
	\displaystyle\max_{i\in \IN}|x_i|&\colon  p=\infty,
\end{cases}
\]
and write $\ell_p:=\{x\in\IR^{\IN}\colon \|x\|_p<\infty\}$ and $\ell_p^n$ for $(\IR^n,\|\cdot\|_p)$.

We now present the asymptotics for the entropy numbers of embeddings between $\ell_p$-spaces.  For $0 < p\le r\le \infty$, one has \begin{equation} \label{eq:classical-p-le-r}
e_k(\id\colon \ell_p^n\to \ell_r^n)
\asymp
\begin{cases}
	1 &\colon k\le \log n,\\
	\Big(\frac{\log(n/k+1)}{k}\Big)^{1/p-1/r}&\colon  \log n\le k\le n,\\
	2^{-k/n}n^{1/r-1/p}&\colon k\ge n,
\end{cases}
\end{equation}
and if $0< r\le p\le \infty$, then
\begin{equation}\label{eq:classical-p-ge-r}
e_k(\id\colon \ell_p^n\to \ell_r^n)
\asymp 2^{-k/n}n^{1/r-1/p}.
\end{equation}
Here, the relation $\asymp$ denotes equivalence up to implicit constants independent of $k$ and $n$ (while they may depend on the parameters $p$ or $r$), and we interpret $1/\infty=0$; for the non-commutative counterpart
to the previous result, we refer to \cite{HPV2017}. 

As mentioned above, for various reasons it is of interest to understand the asymptotic behavior of entropy numbers of finite-dimensional embeddings. 
In the case of embeddings between Orlicz sequence spaces, an asymptotic characterization similar to the one presented before was obtained by Kaewtem and Netrusov \cite[Theorem 4.2]{KN21}.

The main aim of our paper is to give a complete asymptotic characterization of entropy numbers of embeddings of Lorentz sequence spaces. Before we come to that, we introduce the necessary notation
and provide some historical remarks. Lorentz spaces of measurable functions were introduced by G. G. Lorentz \cite{L1950, L1951} and since then they have become an indisposable tool
in mathematical analysis \cite{stein1971}. 
Lorentz spaces arise from Lebesgue spaces via interpolation and we shall give more details on this later in Section~\ref{sec:interpolation}. Moreover, beyond being studied in the functional analysis
literature, Lorentz spaces also play fundamental roles in applied mathematics, for instance, in signal processing \cite{CT06}. In particular, the weak $\ell_p$-spaces $\ell_{p,\infty}^n$ are used
in the theory of compressed sensing \cite{FR13}. 

The theory of Lorentz function spaces includes as a special case also the Lorentz sequence spaces, which we consider in this paper.
We give their definition using the notion of non-increasing rearrangement. If $x=(x_i)_{i\in\IN}\in\IR^{\IN}$ is an infinite sequence, we define its non-increasing rearrangement $x^*=(x^*_i)_{i\in\IN}$, where
$x_i^*:=\inf\{\lambda>0\colon |\{k\in\IN\colon |x_k|> \lambda\}|\le i-1\}$.
For $0< p, u\le \infty$ the Lorentz $\ell_{p,u}$-quasi-norm of $x=(x_i)_{i\in\IN}\in\IR^{\IN}$ is defined as 
\[
\|x\|_{p,u}:=\|i^{1/p-1/u}x_i^*\|_u,
\]
(see, e.g., \cite[(1.4.9)]{Grafakos} or \cite[Lemma 2.9]{LN2023} for the fact that this is a quasi-norm).
Lorentz sequence spaces (at least in the case $1\le u\le p$) appear already in \cite[Section I.3.a]{LiTz} as an example of Banach spaces with a symmetric basis.
We shall write $\ell_{p,u}:=\{x\in\IR^{\IN}\colon \|x\|_{p,u}<\infty\}$ for the corresponding Lorentz sequence space
and $\ell_{p,u}^n$ for the space $(\IR^n,\|\cdot\|_{p,u})$. The finite-dimensional unit ball is then given by $B_{p,u}^n:=\{x\in\IR^n\colon \|x\|_{p,u}\le 1\}$.

For general background on Lorentz sequence spaces, we refer the reader to \cite{AEP1988,CL2019} and the original work of Lorentz \cite{L1951}.
From the point of view of geometric functional analysis, Lorentz sequence spaces form a generalization of $\ell_p$-spaces belonging to the important class of $1$-symmetric Banach spaces.
Various analytic and geometric properties of Lorentz spaces have been studied in the local theory of Banach spaces and geometric functional analysis
(see, e.g., \cite{A2023,DV20,F2023,KPS23,KM04,LN2023,P2021,R1982,Sch1989}).

Let us now elaborate on the relation of Lorentz spaces and entropy numbers, bringing both concepts together. Using finite-dimensional Lorentz space embeddings, Edmunds and Netrusov \cite{EN11} disproved a conjecture regarding the interpolation behavior of entropy numbers. More precisely, they showed that entropy numbers are not compatible with respect to interpolation on both sides, while interpolation on either side is indeed possible. Let us give some explanation (we refer to Section~\ref{sec:interpolation} below for more details). Given pairs $(X_0,X_1)$ and $(Y_0,Y_1)$ of Banach spaces which are each embedded into a common Hausdorff topological space, a linear operator $T\colon X_0+X_1\to Y_0+Y_1$, and parameters $\theta\in (0,1)$ and $1\le u\le \infty$, it was disproved in \cite{EN11} that there is $C\in(0,\infty)$ such that, for all $k_0,k_1\in\IN$,
\[
e_{k_0+k_1-1}(T\colon (X_0,X_1)_{\theta,u}\to (Y_0,Y_1)_{\theta,u})
\le C e_{k_0}^{1-\theta}(T\colon X_0 \to Y_0) e_{k_1}^{\theta}(T\colon X_1 \to Y_1).
\]
A counterexample is provided by a diagonal operator between Lorentz spaces with logarithmically decaying diagonal. 

Having mentioned or referred to a number of results concerning entropy numbers and/or Lorentz spaces, as it turns out, until now, there was no complete picture regarding entropy numbers for Lorentz space embeddings. The main contribution of this work is to close this gap with the following theorem.

\begin{thm}\label{thm:main}
	Let $0<p,q,u,v\le \infty$ and $n\in\IN$. Define the quantity
	\[
	\ell(k,n):=\frac{k}{\log(n/k+1)},\quad \log n\le k \le n.
	\]
	Then the following asymptotics hold:
	\begin{enumerate}
		\item[\emph{(0)}] For $p\neq q<\infty$, we have
		\[
		e_k(\id\colon \ell_{p,u}^n\to\ell_{q,v}^n)\asymp e_k(\id\colon \ell_{p}^n\to\ell_{q}^n), \quad k\in\IN.
		\]
	\item[\emph{(I)}] For $q<p=\infty$, we have
			\[
			e_k(\id\colon \ell_{\infty,u}^n\to \ell_{q,v}^n)
			\asymp 2^{-k/n}n^{1/q}(\log n)^{-1/u},\quad k\in\IN.
			\]
	\item[\emph{(II)}] For $p<q=\infty$, we have 
			\[
			e_k(\id\colon \ell_{p,u}^n\to \ell_{\infty,v}^n)
			\asymp 
			\begin{cases}
				1&\colon k\le \log n,\\
				\ell(k,n)^{-1/p}\log(\ell(k,n))^{1/v}&\colon \log n\le k\le n,\\
				2^{-k/n}n^{-1/p}(\log n)^{1/v}&\colon  k\ge n.
			\end{cases}
			\]
		\item[{\emph{(III)}}] For $p=q<\infty$, we have \\
			{\emph{(III.1)}} whenever $u\le v$
			\[
			e_k(\id\colon \ell_{p,u}^n\to \ell_{p,v}^n)
			\asymp 2^{-k/n},\quad k\in \IN,
			\]
			{\emph{(III.2)}} and whenever $u>v$
			\[
			e_k(\id\colon \ell_{p,u}^n\to \ell_{p,v}^n)
			\asymp
			\begin{cases}
				\log(n/k+1)^{1/v-1/u}&\colon  k\le n,\\
				2^{-k/n}&\colon  k\ge n.
			\end{cases}
			\]
	\item[\emph{(IV)}] For $p=q=\infty$, we have \\
		{\emph{(IV.1)}} whenever $u\ge v$ 
			\[
			e_k(\id\colon \ell_{\infty,u}^n\to \ell_{\infty,v}^n)
			\asymp 2^{-k/n}(\log n)^{1/v-1/u},\quad k\in \IN,
			\]
			{\emph{(IV.2)}} and whenever $u<v$ 
			\[
			e_k(\id\colon \ell_{\infty,u}^n\to \ell_{\infty,v}^n)
			\asymp 
			\begin{cases}
			1&\colon 	 k\le \log n,\\
			\log(\ell(k,n))^{1/v-1/u}&\colon 	\log n\le k\le n,\\
			2^{-k/n}(\log n)^{1/v-1/u}&\colon 	k\ge  n.\\
			\end{cases}
			\]
	\end{enumerate}
	All implicit constants are independent of $k$ and $n$, but may depend on $p,u,q$, or $v$.
\end{thm}

Theorem~\ref{thm:main} shows that, asymptotically, entropy numbers of embeddings between Lorentz spaces exhibit a rich behavior with additional logarithms appearing if $p=q$ or if $p$ or $q$ are infinite.

\begin{remark}\phantom{}
Let us elaborate on previously known results in order to contextualize our contribution.   

\begin{enumerate}
	\item The case $p=u$ and $q=v$ reduces to $\id\colon \ell_p^n\to \ell_q^n$, see \eqref{eq:classical-p-le-r} and \eqref{eq:classical-p-ge-r} above.
	\item The case $0<p\neq q<\infty$ was already stated in \cite[Eq. (29)]{DV20} and also in \cite{EN11}, where the authors actually refer to \cite{EN98}. Since we could not locate a proof in the literature, we provide one ourselves using interpolation. Note that Kaewtem~\cite[Corollary 4.3]{Kae17} proved the special case of $p<q$ and $k\ge n$.
	\item The case $0<p=q=v< \infty$ and $u=\infty$ was proven in \cite[Theorem 10]{DV20}.
\end{enumerate}
\end{remark} 

The following result provides asymptotics for the norm of the natural embedding between Lorentz sequence spaces.  
Using the equivalence between $e_1(T)$ and $\|T\|$ given through the norming property of entropy numbers, it coincides with the choice of $k=1$ in Theorem~\ref{thm:main}.
As we shall need it in the proof of Theorem \ref{thm:main}, we state (and prove) it separately.
Here and in what follows, we write $(x)_+:=\max\{0,x\}$ for the positive part of $x\in\IR$.

\begin{prop}\label{pro:norm-emb}
Let $n\in\N$ and $0<p,q,u,v\le \infty$. We have 
\[
\|\id\colon \ell_{p,u}^n\to \ell_{q,v}^n\|
\asymp 
\begin{cases}
n^{(1/q-1/p)_+}&\colon p\neq q<\infty,\\
n^{1/q}(\log n)^{-1/u}&\colon q<p=\infty,\\
1&\colon p<q=\infty,\\
(\log n)^{(1/v-1/u)_+}&\colon p=q,\\
\end{cases}
\]
where the implicit constants are independent of the dimension $n$.
\end{prop}

In the remainder of this work, we shall present the proof of Theorem~\ref{thm:main} and generalize techniques developed for $\ell_p$-spaces to our Lorentz space setting. As a general rule, bounds for $k\ge n$ and bounds with no case distinction on $k$ are proven with volume techniques, prepared in Section~\ref{sec:volume} more generally for embeddings between quasi-Banach spaces. Upper bounds in the remaining cases are proven via sparse approximation, and monotonicity arguments. Whenever convenient, we shall use interpolation (see Section~\ref{sec:interpolation}). Finally, in Section~\ref{sec:sparse} we state a characterization of entropy numbers of embeddings between finite-dimensional symmetric quasi-Banach spaces in terms of best $s$-term approximation numbers and present an alternative proof of Theorem~\ref{thm:main}.

\textbf{Notation.} Given sequences $(a_k)_{k\in\IN}$ and $(b_k)_{k\in\IN}$ of non-negative real numbers, we write $a_k\lesssim b_k$ if there exists an implicit constant $C\in(0,\infty)$ such that $a_k\le Cb_k$ for all $k\in\IN$. Similarly, we use $a_k\gtrsim b_k$ if $b_k\lesssim a_k$ and $a_k\asymp b_k$ if additionally $a_k\lesssim b_k$ holds. In the following, implicit constants will never depend on $k$ and $n$ and may depend on parameters $p,q,u$ or $v$. With respect to this notation and due to $\log 1=0$, we want to point out that sometimes it may be necessary to replace $\log n$ by $\log(n+1)$. We omit this for the sake of readability.

\section{Entropy numbers of embeddings between quasi-Banach spaces}\label{sec:volume}

Here and in the following, we provide some background information on quasi-normed spaces. Let $X$ be a linear space. A mapping $\|\cdot\|\colon X\to [0,\infty)$ is called a quasi-norm if it satisfies the axioms of a norm except that the triangle inequality is weakened to
\begin{equation} \label{eq:quasi-norm}
	\|x+y\|\le C(\|x\|+\|y\|)\quad \text{for all }x,y\in X,
\end{equation}
where $C\ge 1$ is some constant. If \eqref{eq:quasi-norm} is replaced by
\begin{equation} \label{eq:p-norm}
	\|x+y\|^p\le \|x\|^p+\|y\|^p\quad \text{for all }x,y\in X
\end{equation}
for some $0<p\le 1$, then $\|\cdot\|$ is called a $p$-norm. It follows from Hölder's inequality that every $p$-norm is a quasi-norm with constant $C=2^{1/p-1}$. In fact, by the Aoki-Rolewicz theorem \cite{Aok42,Rol57} every quasi-norm with constant $C\ge 1$ is equivalent to a $p$-norm with $0<p\le 1$ chosen to satisfy $C=2^{1/p-1}$. We say that two quasi-norms $\|\cdot\|_X$ and $\|\cdot\|_Y$ on $X$ are equivalent if and only if there exist $c,C\in(0,\infty)$ such that
\[
c\|x\|_X\le \|x\|_Y\le C\|x\|_X\quad \text{for all }x\in X. 
\]
Note that in this case we may replace $\|\cdot\|_X$ by $\|\cdot\|_Y$ in entropy estimates at the cost of multiplicative constants. Whenever we endow $X$ with a quasi-norm ($p$-norm) and $X$ is complete with respect to the induced distance, it is called a quasi-Banach space ($p$-Banach space). It is useful to note that every $p$-Banach space is also an $r$-Banach space whenever $0<r<p\le 1$, simply because $(a^p+b^p)^{1/p}\le (a^r+b^r)^{1/r}$ for $a,b\ge 0$.

A basis $\{e_1,e_2,\dots\}$ of a quasi-Banach space $(X,\|\cdot\|_X)$ is called $1$-unconditional (or just unconditional) if, for all $x=\sum_{i=1}^{\infty}a_i e_i\in X$ and all sequences of signs $\varepsilon_1,\varepsilon_2,\dots \in\{-1,1\}$ it holds that
\[
\Big\|\sum_{i=1}^{\infty}a_ie_i\Big\|_{X}
=\Big\|\sum_{i=1}^{\infty}\varepsilon_i a_ie_i\Big\|_{X}
\]
and $1$-symmetric (or just symmetric) if, moreover, for all permutations $\pi$ of $\IN$ it holds that
\[
\Big\|\sum_{i=1}^{\infty}a_ie_i\Big\|_{X}
=\Big\|\sum_{i=1}^{\infty}\varepsilon_i a_{\pi(i)}e_i\Big\|_{X}.
\]
The associated fundamental function is defined by
\[
\varphi_X(n)
:=\Big\|\sum_{i=1}^{n}e_i\Big\|_X,\quad n\in\IN.
\]
We shall say that a quasi-Banach space is symmetric if it admits a symmetric basis.

In the following, we shall study entropy numbers of embeddings between $n$-dimensional symmetric quasi-Banach spaces. For this purpose we will need the following monotonicity/lattice property which also holds in the case of an unconditional basis. Its proof was kindly provided to us by G.~Schechtman.

\begin{lem}\label{lem:lattice}
For every quasi-Banach space $X$ with an unconditional basis $\{e_i\}_{i\in\N}$, any $n\in\IN$ and all scalars $a_1,\dots,a_n$ and $b_1,\dots,b_n$ satisfying $|b_i|\le |a_i|$ for all $1\le i\le n$, we have
\[
\Big\| \sum_{i=1}^{n}b_i e_i\Big\|_X
\le K_X\Big\| \sum_{i=1}^{n}a_i e_i\Big\|_X,
\]
where $K_X\ge 1$ depends only on the quasi-norm constant.
\end{lem}
\begin{proof}
By the Aoki-Rolewicz theorem there exists $0<p\le 1$ and a corresponding $p$-norm $\|\cdot\|$ on $X$ which is equivalent to $\|\cdot\|_X$. For $n\in\IN$ let $a_1,\dots,a_n$ and $b_1,\dots,b_n$ be scalars such that $|b_i|\le |a_i|$ for all $1\le i\le n$. Because of the quasi-norm equivalence, it is sufficient to show that
\begin{equation} \label{eq:monotonicity-pnorm}
\Big\|\sum_{i=1}^{n}b_ie_i\Big\|
\le C_p\Big\|\sum_{i=1}^{n}a_ie_i\Big\|,
\end{equation}
where $C_p\in(0,\infty)$ depends only on $p$.  For any $i\in\{1,\dots,n\}$, let us write 
  \[
    |b_i|=|a_i|\sum_{j=1}^{\infty}\delta_{ij} 2^{-j}
  \]   
for some suitable sequence $\delta_{ij}\in\{0,1\}$, $j\in\mathbb N$. Then, for all $j\in\mathbb N$, 
\begin{align*}
	\Big\|2^{-j}\sum_{i=1}^{n}a_i\delta_{ij}e_i\Big\|
	\le 2^{-j}\Big(\frac{1}{2^p}+\frac{1}{2^p}\Big)^{1/p}\Big\|\sum_{i=1}^{n}a_ie_i\Big\|,
\end{align*}
where we used that $\delta_{ij}=\frac{1}{2}\varepsilon_{ij}+\frac{1}{2}$ for some $\varepsilon_{ij}\in \{-1,1\}$ and that $\|\cdot\|$ can also be assumed to be unconditional (see e.g.~the proof of~\cite[Proposition~1.c.5]{K1986}). Therefore, we obtain
\begin{align*}
\Big\|\sum_{i=1}^{n}b_ie_i\Big\|
&= \Big\|\sum_{j=1}^{\infty}2^{-j}\sum_{i=1}^{n}a_i\delta_{ij}e_i\Big\|\\
&\le 2^{\frac{1-p}{p}}\Big(\sum_{j=1}^{\infty}2^{-jp}\Big)^{1/p}\Big\|\sum_{i=1}^{n}a_ie_i\Big\|,
\end{align*}
which yields \eqref{eq:monotonicity-pnorm} with $C_p:= \frac{1}{2}\Big(\frac{2}{2^p-1}\Big)^{1/p}$. This completes the proof.
\end{proof}

\begin{remark}
If $X$ is a Banach space, then Lemma~\ref{lem:lattice} holds with $K_X=1$. This follows from \cite[Theorem 2]{BSW61} or \cite[Proposition 3.1.3]{AK06}. We further remark that Lemma~\ref{lem:lattice} is closely related to the so-called \emph{lattice property} of (quasi-)Banach spaces, which is widely used in functional analysis, see, e.g., \cite[Section 13.1]{Cal64} or \cite[(P2) in Definition 1.1.1]{BS88}. 
\end{remark}

Let $n\in\IN$ and $X, Y$ be $n$-dimensional quasi-Banach spaces with normalized symmetric bases $\{e_i\}_{i=1}^n$ and $\{f_i\}_{i=1}^n$, respectively. We will study the behavior of the entropy numbers of the embedding
\[
\id\colon X\to Y, \quad \id\Big(\sum_{i=1}^{n}x_ie_i\Big)=\sum_{i=1}^{n}x_if_i, \quad (x_1,\dots,x_n)\in\IR^n.
\]
For this we will use the following elementary lemma relating the operator norm of the natural identity between $\ell_\infty^n$ and a symmetric quasi-Banach space $X$ with the fundamental function of the space.

\begin{lem}\label{lem:monotone-fundamental}
Let $n\in\N$ and $X$ be an $n$-dimensional quasi-Banach space with a symmetric basis $\{e_i\}_{i=1}^n$. Then
\[
\varphi_X(n)
\le \|\id\colon \ell_{\infty}^n\to X\|
\le K_X \varphi_X(n),
\]
where we identify $\ell_{\infty}^n=({\rm span}\{e_1,\dots,e_n\},\|\cdot\|_{\infty})$ and $K_X\geq 1$ is as in Lemma~\ref{lem:lattice}.
\end{lem}
\begin{proof}
By Lemma~\ref{lem:lattice} it holds that
\[
\|\id\colon \ell_{\infty}^n\to X\|
=\sup_{x\in B_{\infty}^n}\Big\|\sum_{i=1}^{n}x_i e_i\Big\|_X
\le K_X \Big\|\sum_{i=1}^{n}e_i\Big\|_X,
\]
and for the lower bound we specify $x=(1,\dots,1)$.
\end{proof}

The following proposition taken from \cite[Theorem~4.2]{Kae17} generalizes \cite[Lemma 4]{Sch84} to quasi-Banach spaces. 

\begin{prop}\label{pro:fundamental}
Let $n\in\IN$ and $X, Y$ be $n$-dimensional quasi-Banach spaces with quasi-norm constants $C_X,C_Y\ge 1$ and normalized symmetric bases $\{e_i\}_{i=1}^n$ and $\{f_i\}_{i=1}^n$, respectively. Then
\begin{equation} \label{eq:ent-fundamental}
e_k(\id\colon X\to Y)\asymp 2^{-k/n}\frac{\varphi_Y(n)}{\varphi_X(n)},\quad k\ge n,
\end{equation}
where the implicit constants depend only on $\max\{C_{X},C_{Y}\}$.
\end{prop}

\begin{proof}
In order to derive the statement from \cite[Theorem~4.2]{Kae17}, which is in terms of $p$-Banach spaces, we note that by the Aoki-Rolewicz theorem, we find $0<p,q\le 1$, a $p$-norm $|||\cdot|||_p$ and a $q$-norm $|||\cdot|||_q$ equivalent to $\|\cdot\|_X$ and $\|\cdot\|_Y$, respectively. Then both $|||\cdot|||_p$ and $|||\cdot|||_q$ are $r$-norms with $r=\min\{p,q\}$ depending only on $\max\{C_X,C_Y\}$ and we can apply \cite[Theorem~4.2]{Kae17}. Switching back to the original quasi-norms incurs additional implicit constants depending only on $r$.
\end{proof}

We note that the proof of \cite[Theorem~4.2]{Kae17} implicitly uses the statement of Lemma~\ref{lem:lattice}. Moreover, it essentially involves the following asymptotic inequality from \cite[Section 4, Lemma 3 (ii)]{EN98} which states that
\begin{equation} \label{eq:en-fundamental}
e_n(\id\colon X\to \ell_{\infty}^n)
\lesssim  \varphi_X(n)^{-1},
\end{equation}
where $X$ is as in the assumption of Proposition~\ref{pro:fundamental} and the implicit constant depends only on $C_X$. Note that in the proof of \eqref{eq:en-fundamental} the authors of \cite{EN98} crucially use symmetry. We can use \eqref{eq:en-fundamental} to prove the following generalization of Schütt's result \cite[Lemma 3]{Sch84} to quasi-Banach spaces, which was used in \cite{Sch84} to prove \eqref{eq:ent-fundamental} in the case of Banach spaces. 

\begin{prop}\label{pro:fundamental-vol}
Let $n\in\N$ and $X$ an $n$-dimensional quasi-Banach space with normalized symmetric basis $\{e_i\}_{i=1}^n$. Then
\[
\varphi_X(n)^{-1}\asymp \vol(B_X)^{1/n}, 
\]
where $\vol$ denotes Lebesgue measure on $\IR^n$, which is identified with ${\rm span}\{e_1,\dots,e_n\}$. The implicit constants depend only on the quasi-norm constant.
\end{prop}

Proposition~\ref{pro:fundamental-vol} allows us to rewrite the bounds in Proposition~\ref{pro:fundamental} to 
\begin{equation} \label{eq:vol-large}
	e_k(\id\colon X\to Y) \asymp 2^{-k/n}{\rm rv}(X,Y),\quad k\ge n,
\end{equation}
where
\[
{\rm rv}(X,Y):=\frac{\vol(B_X)^{1/n}}{\vol(B_Y)^{1/n}}.
\]
Here, we identify $B_X$ with $\id(B_X)\subset Y$ and $Y$ with $\IR^n$ using a suitable basis.
Note that ${\rm rv}(X,Y)$ is the normalized ratio of volumes of the unit balls of $X$ and $Y$, respectively, and that it differs from the notion of \emph{volume ratio}, used in the local theory of Banach spaces.

For the proof of Proposition~\ref{pro:fundamental-vol}, we shall use volume comparison arguments and volume bounds such as the following lower bound on entropy numbers by the normalized ratio of volumes.

\begin{lem}\label{lem:vol-lower}
Let $n\in\N$ and $X$, $Y$ be $n$-dimensional quasi-Banach spaces. Then, for any $k\in\IN$,
\[
e_k(\id\colon X\to Y) \ge 2^{-\frac{k-1}{n}}{\rm rv}(X,Y).
\]
\end{lem}
\begin{proof}
Suppose that $B_X$ is covered by $2^{k-1}$ balls of radius $r>0$ in the space $Y$ for some $k\in\IN$. Then a union bound immediately gives 
\[
\vol(B_X)\le 2^{k-1}r^n\vol(B_Y).
\]
Thus, $r\ge 2^{-(k-1)/n}{\rm rv}(X,Y)$, and the result follows.
\end{proof}

\begin{proof}[Proof of Proposition~\ref{pro:fundamental-vol}]

In the following, $\ell_{\infty}^n$ is taken with respect to the basis $\{e_i\}_{i=1}^n$. We conclude from Lemma~\ref{lem:vol-lower} and the inequality \eqref{eq:en-fundamental} that
\[
\frac{1}{2}\vol(B_X)^{1/n}
={\rm rv}(X,\ell_{\infty}^n)
\le 2 e_n(\id\colon X\to \ell_{\infty}^n)
\lesssim \varphi_X(n)^{-1},
\]
which completes the proof of the lower bound.

For the upper bound, we shall use a volume comparison argument. Consider the vectors 
\[
y_{\varepsilon}=\frac{\sum_{j=1}^{n}\varepsilon_je_j}{\varphi_X(n)} ,\quad \varepsilon=(\varepsilon_j)_{j=1}^n\in \{-1,1\}^n.
\]
Then $\|y_{\varepsilon}\|_X=1$ and $\|y_{\varepsilon}-y_{\varepsilon'}\|_{\infty}\ge \frac{2}{\varphi_X(n)}$ for each $\varepsilon\neq \varepsilon'$. Therefore, the balls $y_{\varepsilon}+\varphi_X(n)^{-1}B_{\infty}^n$, $\varepsilon\in\{-1,1\}^n$, are disjoint, and if $z\in y_{\varepsilon}+\varphi_X(n)^{-1}B_{\infty}^n$ for some $\varepsilon$, then, by Lemma~\ref{lem:monotone-fundamental}, 
\[
\|z\|_X
\le C_X(\|y_{\varepsilon}\|_X + \|z-y_{\varepsilon}\|_X)
\le C_X(1 + \|\id\colon \ell_{\infty}\to X \|\|z-y_{\varepsilon}\|_{\infty})
\le c_X,
\]
where $c_X=C_X(1+K_X)\le 2C_X K_X$ with $C_X$ being the quasi-norm constant of $\|\cdot\|_X$ and $K_X$ as in Lemma~\ref{lem:monotone-fundamental}.
So the disjoint balls $y_{\varepsilon}+\varphi_X(n)^{-1}B_{\infty}^n$, $\varepsilon\in\{-1,1\}^n$, are contained in $c_X B_{X}$. Hence, a comparison of volumes shows that
\[
 2^n\varphi_X(n)^{-n}\vol(B_{\infty}^n) \le  c_X^n \vol(B_X),
\]
which is equivalent to
\[
\varphi_X(n)^{-1} (4/c_X )\le \vol(B_X)^{1/n}.
\]
This concludes the proof.
\end{proof}

We show that under some additional assumption, \eqref{eq:vol-large} may in fact be extended to all $k$'s.

\begin{prop}\label{pro:vol-upper}
Let $n\in\IN$ and $X$, $Y$ be $n$-dimensional quasi-Banach spaces with symmetric bases. If 
\begin{equation} \label{eq:upper-1}
	\|\id\colon X\to Y\| \lesssim  {\rm rv}(X,Y),
\end{equation}
then
\[
e_k(\id\colon X\to Y) \asymp 2^{-k/n}{\rm rv}(X,Y),\quad k\in\IN.
\]
The implicit constants do not depend on $k$ or $n$.
\end{prop}
\begin{proof}
If $k\ge n$, this follows from \eqref{eq:vol-large}. If $k\le n$, then by monotonicity
\[
e_k(\id\colon X\to Y)
\le \|\id\colon X\to Y\|
\le 2\cdot 2^{-k/n}\|\id\colon X\to Y\|.
\]
Together with \eqref{eq:upper-1}, this gives the upper bound. Finally, the lower bound follows from Lemma \ref{lem:vol-lower}.
\end{proof}

Note that we always have
\begin{equation}\label{eq:id-vol}
\|\id\colon X\to Y\|\ge {\rm rv}(X,Y).
\end{equation}
In particular, we can replace \eqref{eq:upper-1} by $\|\id\colon X\to Y\| \asymp  {\rm rv}(X,Y)$. For convenience of the reader, we provide a proof.

\begin{proof}[Proof of \eqref{eq:id-vol}]
Write
\[
\|\id\colon X\to Y\|
=\sup_{\|y\|_X\le 1}\|y\|_Y
=\inf\{r>0\colon B_X\subset rB_Y\}.
\]
If $B_X\subset rB_Y$ for some $r>0$, then we have $\vol(B_X)\le r^{n}\vol(B_Y)$, that is, ${\rm rv}(X,Y)\le r$. So if $\|\id\colon X\to Y\|\le r$, then ${\rm rv}(X,Y)\le r+\varepsilon$ for every $\varepsilon>0$, which proves the statement.
\end{proof}

\begin{remark}
For completeness we remark that the conclusion of Proposition~\ref{pro:vol-upper} under \eqref{eq:upper-1} can be shown directly for all $k\in\IN$. For this, note that by \cite[Lemma~2.1]{HKV16}, for an $n$-dimensional $p$-Banach space $X$ and for $k\in\IN$, we have
\begin{equation} \label{eq:entropy-identity}
e_k(\id\colon X\to X)\le 4^{1/p}2^{-(k-1)/n}.
\end{equation}
By factorization and \eqref{eq:entropy-identity}, we have for an $n$-dimensional $p$-Banach space $X$ and a quasi-Banach space $Y$ that
\begin{align*}
e_k(\id\colon X\to Y)
&\le e_k(\id\colon X\to X)\|\id\colon X\to Y\|\\
&\le 4^{1/p}2^{-\frac{k-1}{n}}\|\id\colon X\to Y\|.
\end{align*}
Using \eqref{eq:upper-1}, Lemma \ref{lem:vol-lower} and the Aoki-Rolewicz theorem completes the proof.
\end{remark}

\subsection{Interpolation}\label{sec:interpolation}

We already mentioned that the Lorentz sequence space $\ell_{p,u}$ arises from real interpolation of $\ell_p$-spaces. For convenience of the reader, we give more details on this procedure and refer to \cite{BL76} for more information. 

Let $(X_0,X_1)$ be a pair of quasi-normed spaces such that there is a quasi-normed space $\mathcal{X}$, in which both spaces are continuously embedded. Let $X_0+X_1$ be the space of all $x\in\mathcal{X}$ with $x=x_0+x_1$ for $x_i\in X_i$, $i\in\{0,1\}$. Define for $x\in X_0+X_1$ the $K$-functional by
\[
K(t,x):=\inf\big\{ \|x_0\|_{X_0}+t\|x_1\|_{X_1}\colon x=x_0+x_1\text{ with } x_i\in X_i, i\in\{0,1\}\big\},\quad t>0.
\]
Let $0<\theta<1$ and $0<u\le \infty$. Then $(X_0,X_1)_{\theta,u}$ is the space of all $x\in X_0+X_1$ such that  
\[
\|x\|_{(\theta,u)}
:=
\begin{cases}
\Big(\int_0^{\infty} (t^{-\theta}K(t,x))^u\frac{\dd t}{t}\Big)^{1/u}&\colon u<\infty,\\
\sup_{t>0} t^{-\theta}K(t,x)&\colon u=\infty,\\
\end{cases}
\]
is finite. The space is endowed with the quasi-norm $\|\cdot\|_{\theta,u}$. Note that if one of the spaces $X_0$ or $X_1$ is continuously embedded into the other, they automatically form a pair as above. This is the case with $\ell_p$-spaces and also $\ell_{p,u}$-spaces. The following result is taken from \cite[Theorem 5.3.1]{BL76}.

\begin{prop}\label{pro:reiter}
Let $0<p_0,p_1,u_0,u_1,p,u\le \infty$. If $p_0\neq p_1$ and $1/p=(1-\theta)/p_0+\theta/p_1$ for some $\theta \in (0,1)$, then
\[
(\ell_{p_0,u_0},\ell_{p_1,u_1})_{\theta,u}=\ell_{p,u}.
\]
Moreover, the quasi-norms $\|\cdot\|_{(\theta,u)}$ and $\|\cdot\|_{p,u}$ are equivalent. This statement remains true if $p_0=p_1=p$, provided that $1/u=(1-\theta)/u_0+\theta/u_1$.
\end{prop}

Entropy numbers behave well with respect to interpolation on either side, but not on both (as we mentioned before). The following result is adapted from \cite[Theorem 1.3.2]{ET96}.

\begin{prop}\label{pro:entropy-interpolation}
Let $Y$ be a quasi-Banach space and $(X_0,X_1)$ be a pair as above, $\theta\in (0,1)$ and $0<u\le \infty$.  
\begin{enumerate}
\item If $T\colon Y\to X_0\, \cap\, X_1$ is linear and continuous with respect to $\|x\|=\max\{\|x\|_{X_0},\|x\|_{X_1}\}$, $x\in X_0\cap X_1$, then, for all $k_0,k_1\in\IN$, we have
\[
e_{k_0+k_1-1}(T\colon Y\to (X_0,X_1)_{\theta,u})
\le C e_{k_0}^{1-\theta}(T\colon Y \to X_0) e_{k_1}^{\theta}(T\colon Y \to X_1).
\]
\item If $T\colon X_0+X_1\to Y$ is linear such that its restrictions to $X_0$ and $X_1$ are continuous, then, for all $k_0,k_1\in\IN$, we have
\[
e_{k_0+k_1-1}(T\colon (X_0,X_1)_{\theta,u}\to Y)
\le C e_{k_0}^{1-\theta}(T\colon X_0\to Y) e_{k_1}^{\theta}(T\colon  X_1\to Y).
\]
\end{enumerate}
Here, the constant $C\in(0,\infty)$ depends only on the quasi-norm constants of $X_0$ and $X_1$.
\end{prop}

We shall apply the statements of this section to prove case (0) in Theorem~\ref{thm:main}. Note that with regard to Proposition~\ref{pro:reiter}, the restriction to the first $n$ coordinates does not change results and that it is sufficient to consider equivalent quasi-norms. 

\section{The size of the unit ball of a Lorentz space}

As a preparation for the proof of Theorem~\ref{thm:main}, we prove asymptotics for the volume of the unit ball of a Lorentz space (Lemma~\ref{lem:volume}) and for its size when measured in the quasi-norm of another Lorentz space (Proposition~\ref{pro:norm-emb}); such results are of independent interest.

The fundamental function of $\ell_{p,u}$ with $0<p,u\le \infty$ satisfies
\begin{equation} \label{eq:fundamental-lorentz}
	\varphi_{\ell_{p,u}}(n)
\asymp
\begin{cases}
	n^{1/p}&\colon p<\infty,\\
	(\log n)^{1/u} &\colon p=\infty.
\end{cases}
\end{equation}

Combined with Proposition~\ref{pro:fundamental-vol} this yields the following asymptotics for the volume of Lorentz balls. The case $p<\infty$ can be found in \cite[Theorem 7]{DV20} and is proven using interpolation methods. In the case of $p=1$, the volume of $B_{p,u}^n$ can be computed explicitly and precise asymptotics become available, see \cite[Theorem 5]{DV20} and \cite[Corollary 1]{KPS23}. 

\begin{lem}\label{lem:volume}
For all $0<p,u\le \infty$, we have 
\[
\vol(B_{p,u}^n)^{1/n}\asymp 
\begin{cases}
n^{-1/p}  &\colon p<\infty,\\
(\log n)^{-1/u}  &\colon p=\infty.
\end{cases}
\]
\end{lem}

For convenience of the reader we give a direct proof in the case of $p=\infty$.

\begin{proof}[Proof of Lemma~\ref{lem:volume} for $p=\infty$ and $u<\infty$]
Let us denote
\begin{equation}\label{def:Hn}
H_n:=\sum_{k=1}^n k^{-1}.
\end{equation}
Then $H_n$ grows logarithmically in $n$ and  $H_n^{-1/u}\cdot [-1,1]^n\subset B^n_{\infty,u}$, which gives the lower bound.

To show the upper bound, we fix some $c>1$ and denote by $K\le n$ the maximal number of indices of $x\in B^n_{\infty,u}$, where $|x_j|> cH_n^{-1/u}$. Then
\[
1\ge \sum_{k=1}^n k^{-1} (x_k^*)^u\ge (x^*_K)^u\sum_{k=1}^K k^{-1}\ge c^u\cdot  H_n^{-1}\cdot H_K.
\]
Letting $c:=6^{1/u}$ and using the elementary estimate $\log n\le H_n \le 3\log n$, we obtain $K\le \sqrt{n}$.

We can now cover $B^n_{\infty,u}$ by the union of $\binom{n}{K}$ cubes having sides $[-1,1]$ in exactly $K$ coordinates and $[-cH_n^{-1/u},cH_n^{-1/u}]$ in the remaining ones.
By volume comparison, we obtain
\[
\vol(B_{\infty,u}^n)\le \binom{n}{K}\vol([-1,1]^K\times [-cH_n^{-1/u},cH_n^{-1/u}]^{n-K})
\]
and
\begin{align*}
\vol(B_{\infty,u}^n)^{1/n}&\le \binom{n}{K}^{1/n}\cdot 2^{K/n}\cdot (2cH_n^{-1/u})^{1-K/n}\\
&\le 2\cdot 2\cdot (2c H_n^{-1/u})\cdot (2c)^{-K/n}\cdot H_n^{K/(un)}\\
&\le 8c H_n^{-1/u}\cdot H_n^{K/(un)}.
\end{align*}
Finally, we observe that $H_n^{K/(un)}$ is bounded due to $K\le \sqrt{n}$.
\end{proof}

We shall need the following decay estimates for the largest entries. These are essentially sharp as shown by $x=\mathbf{1}$ for $p<\infty$ and $x=( (\log i)^{-1/u})_{i=1}^n$ for $p=\infty$.
\begin{lem}\label{lem:x-star-bound}
Let $n\in\IN$. For all $x\in\IR^n$ and $i\in \{1,\dots,n\}$, we have  
\[
x_i^*\lesssim \|x\|_{p,u}
\begin{cases}
	i^{-1/p}&\colon  p<\infty,\\
	(\log i)^{-1/u}&\colon  p=\infty.\\
\end{cases}
\]
\end{lem}
\begin{proof}
If $p=u=\infty$, then this trivially holds. If $p<u=\infty$, then
\[
x_i^*
=i^{-1/p}i^{1/p}x_i^*
\le i^{-1/p}\max_{1\le j\le n}j^{1/p}x_j^*
=i^{-1/p}\|x\|_{p,\infty}.
\]
If $u<\infty$, then we proceed as follows. There are $z_{1},\dots,z_{n}\ge 0$ such that $(x_i^*)^u=\sum_{\ell=i}^{n}z_{\ell}$ for every $i\in\{1,\dots, n\}$. If $p<\infty$, then for $\beta=u/p$,
\begin{align*}
	i^{\beta}\sum_{\ell=i}^{n}z_{\ell}
\le \sum_{\ell=1}^{n}z_{\ell}\ell^{\beta}
\lesssim \sum_{\ell=1}^{n}z_{\ell}\sum_{j=1}^{\ell}j^{\beta-1}
= \sum_{j=1}^{n}j^{\beta-1}\sum_{\ell=j}^{n}z_{\ell}
=\|x\|_{p,u}^u.
\end{align*}
If $p=\infty$, i.e., $\beta=0$, then this remains valid if we replace $i^{\beta}$ and $\ell^\beta$ by $\log i$ and $\log \ell$, respectively.
\end{proof}

Next, we prove Proposition~\ref{pro:norm-emb}.
We shall use that for $0<p\le \infty$ and $0<u\le v\le \infty$ there exists a constant $c_{p,u,v}\in(0,\infty)$ such that 
\begin{equation} \label{eq:ordering}
\|x\|_{p,v}\le c_{p,u,v}\|x\|_{p,u}\quad x\in\IR^n.
\end{equation}
This is a well known fact, see \cite[Proposition 4.2]{BS88} and also \cite[Proposition 6]{DV20} and its proof.

\begin{proof}[Proof of Proposition~\ref{pro:norm-emb}]
In what follows, we let $x\in \IR^n$.
\vskip 2mm

\noindent\emph{Case 1.} Let $p\not=q<\infty$. The lower bound follows by choosing $x=e_1\in B^n_{p,u}$ if $p<q$ and 
$x/\|x\|_{p,u}\in B^n_{p,u}$ with $x=\sum_{i=1}^n e_i$ if $q<p.$

To show the upper bound, we observe that by \eqref{eq:ordering} it is enough to consider only $u=\infty.$ In this case, we estimate
\begin{align*}
\|x\|_{q,v}^v&=\sum_{i=1}^n i^{v/q-1}\cdot i^{-v/p} \cdot i^{v/p}\cdot (x_i^*)^v\le \max_{1\le j\le n} j^{v/p} (x_j^*)^v\cdot \sum_{i=1}^n i^{v/q-v/p-1}\\
&\lesssim \|x\|_{p,\infty}^v\cdot n^{v(1/q-1/p)_+}
\end{align*}
if $v<\infty$ and
\begin{align*}
\|x\|_{q,\infty}=\max_{1\le j\le n} j^{1/q}x_j^* = \max_{1\le j\le n} j^{1/q-1/p}\cdot j^{1/p}x_j^*\le \|x\|_{p,\infty}\cdot n^{v(1/q-1/p)_+}
\end{align*}
if $v=\infty.$

\vskip 2mm
\noindent\emph{Case 2.} Let $q<p=\infty$ and $0<u,v\le \infty$. The lower bound is obtained by choosing $x/\|x\|_{\infty,u}\in B_{\infty,u}^n$ with $x:=\sum_{i=1}^{n}e_i$. For the upper bound, we first assume that $v<\infty$.
Then, by Lemma \ref{lem:x-star-bound},
\[
\|x\|_{q,v}^v
=\sum_{i=1}^{n}i^{v/q-1}(x_i^*)^v
\le \|x\|_{\infty,u}\sum_{i=1}^{n}i^{v/q-1}(\log i)^{-v/u}.
\]
To complete the proof of the upper bound in this case, we use the known asymptotics
\[
\sum_{i=1}^{n}i^{\lambda-1}(\log i)^{\beta}
\asymp n^{\lambda}(\log n)^{\beta},
\]
valid for $\lambda>0$ and $\beta\in \IR$ with implicit constants independent of $n$. Now assume that $v=\infty$. Then
\[
\|x\|_{q,\infty}
=\max_{1\le i\le n}i^{1/q}x_i^*
\le \|x\|_{\infty,u}\max_{1\le i\le n}i^{1/q}(\log i)^{-1/u}
\lesssim n^{1/q}(\log n)^{-1/u}\|x\|_{\infty,u}.
\]
\vskip 2mm
\noindent\emph{Case 3.} Let $p<q=\infty$ and $0<u,v\le \infty$. The lower bound is obtained by choosing the vector $x=e_1\in B_{p,u}^n$. First, let $v<\infty$. Then
\[
\|x\|_{\infty,v}^v
=\sum_{i=1}^{n}i^{-1}(x_i^*)^v
\le \|x\|_{p,u}\sum_{i=1}^{n}i^{-v/p-1}
\lesssim \|x\|_{p,u}.
\]
Now let $v=\infty$. Then, by the estimate in \eqref{eq:ordering}, we have $\|x\|_{\infty}\lesssim \|x\|_{\infty,v}$. This completes the proof of the upper bound.

\noindent\emph{Case 4.} Let $p=q\leq \infty$. This case splits into two cases. 
\vskip 1mm
$u\le v$: Then we conclude the upper bound from \eqref{eq:ordering}, while the lower bound simply follows by choosing $x=e_1\in B_{p,u}^n$.
\vskip 1mm
$u>v$:  Then we deduce from Hölder's inequality applied with conjugate indices $r:=u/v>1$ and $r^*:=u/(u-v)$ that
\begin{align*}
\|x\|_{p,v}
& = \Big(\sum_{k=1}^{n}(k^{1/p}x_k^*)^{v}k^{-v/u}\cdot k^{-1+v/u}\Big)^{1/v} \cr
& \le \Big(\sum_{k=1}^{n}(k^{1/p}x_k^*)^{u}k^{-1}\Big)^{1/u}\Big(\sum_{k=1}^{n}k^{-1}\Big)^{(u-v)/uv},
\end{align*}
where the first factor on the right-hand side is just $\|x\|_{p,u}$, while the second factor equals $H_n^{1/v-1/u}$ with $H_n$ having been  introduced in \eqref{def:Hn}.  Therefore, Hölder's inequality immediately gives
  \[
    \|\id\colon \ell_{p,u}^n\to \ell_{p,v}^n\| \leq \sup_{\|x\|_{p,u}\leq 1} \|x\|_{p,u} H_n^{1/v-1/u} = H_n^{1/v-1/u} .
  \]
For the corresponding lower bound, we just observe that $x:=(k^{-1/p})_{k=1}^n$ satisfies $\|x\|_{p,u}=H_n^{1/u}$ as well as $\|x\|_{p,v}=H_n^{1/v}$, and so, because $\|H_n^{-1/u}x\|_{p,u}=1$, it follows that
  \[
    \|\id\colon \ell_{p,u}^n\to \ell_{p,v}^n\| \geq  \|H_n^{-1/u}x\|_{p,v}  = H_n^{1/v-1/u}.
  \]
Thus, we have
   \[
    \|\id\colon \ell_{p,u}^n\to \ell_{p,v}^n\| = H_n^{1/v-1/u} \asymp (\log n)^{1/v-1/u},
  \]
where the latter asymptotic follows directly from the definition of $H_n$.  
\end{proof}

\section{Proof of Theorem~\ref{thm:main}}

We first give a proof of the case $0<p\neq q<\infty$, where we follow the general strategy set out in \cite[Section~4]{DV20}. Essentially, it relies on interpolation properties of Lorentz spaces and entropy numbers,
as detailed in Section~\ref{sec:interpolation}.   

\begin{proof}[Proof of case (0).]
We only present the proof in the case $0<p<q< \infty$ and note that the case $0<q<p< \infty$ can be proven in a similar way.  For the upper bound, let $0<r<p<s<q< \infty$ with $\frac{1}{s}=\frac{1}{2}(\frac{1}{p}+\frac{1}{q})$ and $\frac{1}{p}=\frac{1}{2}(\frac{1}{r}+\frac{1}{s})$. Then, by Proposition~\ref{pro:reiter},
\[
\ell_{p,u}^n=(\ell_{r}^n,\ell_{s}^n)_{\frac{1}{2},u} \quad \text{and}\quad\ell_{q,v}^n=(\ell_{s}^n,\ell_{\infty}^n)_{\theta,v}
\]
with $\theta=1-\frac{s}{q}\in (0,1)$. Therefore, by Proposition~\ref{pro:entropy-interpolation}, for every $k\in\IN$,
\begin{align*}
e_{4k-3}(\id\colon \ell_{p,u}^n\to \ell_{q,v}^n)
&\lesssim e_{2k-1}(\id\colon \ell_r^n\to \ell_{q,v}^n)^{1/2}e_{2k-1}(\id\colon \ell_s^n\to \ell_{q,v}^n)^{1/2}\\
&\lesssim e_k(\id\colon \ell_r^n\to \ell_{s}^n)^{(1-\theta)/2}e_k(\id\colon \ell_r^n\to \ell_{\infty}^n)^{\theta/2}\\
&\ \times e_k(\id\colon \ell_s^n\to \ell_{s}^n)^{(1-\theta)/2}e_k(\id\colon \ell_s^n\to \ell_{\infty}^n)^{\theta/2}.
\end{align*}
Using the upper bounds in \eqref{eq:classical-p-le-r}, we see that 
\[
e_{4k-3}(\id\colon \ell_{p,u}^n\to \ell_{q,v}^n)\lesssim 2^{-k/n}n^{1/q-1/p},
\]
and using monotonicity completes the proof of the upper bound. 

For the lower bound in the case $0<p<q<\infty$ choose $p_2,q_2$ such that $0<p<p_2<q<q_2<\infty$, as well as
$
\frac{1}{p_1}=\frac{1}{2}(\frac{1}{p}+\frac{1}{p_2})\text{ and }\frac{1}{q_1}=\frac{1}{2}(\frac{1}{q}+\frac{1}{q_2})
$ 
such that $0<p<p_1<p_2<q<q_1<q_2<\infty$. Then, by Proposition~\ref{pro:reiter},
\[
\ell_{p_1}^n=(\ell_{p,u}^n,\ell_{p_2}^n)_{\frac{1}{2},p_1} \quad \text{and}\quad\ell_{q_1}^n=(\ell_{q,v}^n,\ell_{q_2}^n)_{\frac{1}{2},q_1}.
\]
Again by Proposition~\ref{pro:entropy-interpolation}, we have, for every $k\in\IN$, that
\begin{align*}
e_{4k-3}(\id\colon \ell_{p_1}^n\to \ell_{q_1}^n)
&\lesssim e_{2k-1}(\id\colon \ell_{p,u}^n\to \ell_{q_1}^n)^{1/2}e_{2k-1}(\id\colon \ell_{p_2}^n\to \ell_{q_1}^n)^{1/2}\\
&\lesssim e_k(\id\colon \ell_{p,u}^n\to \ell_{q,v}^n)^{1/4}e_k(\id\colon \ell_{p,u}^n\to \ell_{q_2}^n)^{1/4}\\
&\ \times e_k(\id\colon \ell_{p_2}^n\to \ell_{q,v}^n)^{1/4}e_k(\id\colon \ell_{p_2}^n\to \ell_{q_2}^n)^{1/4}.
\end{align*}
Using the upper bound we just proved, we obtain 
\begin{align*}
e_{4k-3}(\id\colon \ell_{p_1}^n\to \ell_{q_1}^n)
&\lesssim e_k(\id\colon \ell_{p,u}^n\to \ell_{q,v}^n)^{1/4}e_k(\id\colon \ell_{p}^n\to \ell_{q_2}^n)^{1/4}\\
&\ \times e_k(\id\colon \ell_{p_2}^n\to \ell_{q}^n)^{1/4}e_k(\id\colon \ell_{p_2}^n\to \ell_{q_2}^n)^{1/4}
\end{align*}
since $p\neq q_2$ and $q\neq p_2$. Plugging in the upper bounds from \eqref{eq:classical-p-le-r} and using monotonicity gives the lower bound.
\end{proof}

We now give the proofs of the cases (I), (II), (III) and (IV).

We first treat the cases which follow from volume estimates. 

\begin{proof}[Proof of (I), (III.1), (IV.1) for $k\in\IN$ and of (II), (III.2), (IV.2) for $k\ge n$.]

In all of these cases Theorem~\ref{thm:main} follows for $k\ge n$ from Proposition~\ref{pro:fundamental} and \eqref{eq:fundamental-lorentz}.

	For the proof of (I), (III.1) and (IV.1) also for $k\le n$, we note that by Proposition~\ref{pro:norm-emb} and Lemma~\ref{lem:volume} in each case it holds that
\[
\|\id\colon \ell_{p,u}^n\to\ell_{q,v}^n\|
\asymp {\rm rv}(\ell_{p,u}^n,\ell_{q,v}^n).
\]
Therefore, Lemma~\ref{lem:vol-lower} and Proposition~\ref{pro:vol-upper} imply
\[
e_k(\id\colon \ell_{p,u}^n\to \ell_{q,v}^n) \asymp 2^{-k/n}{\rm rv}(\ell_{p,u}^n,\ell_{q,v}^n),\quad k\in\IN.
\]
\end{proof}

We now prove the bounds for small $k\le n$ in the remaining cases (II), (III.2), and (IV.2). 
To this end, we will employ \cite[Section 4, Theorem 2]{EN98}. Roughly speaking, it characterizes the behavior of $e_k(\id\colon X\to Y)$, $k<n/2$, for $n$-dimensional quasi-Banach spaces $X$ and $Y$ with common symmetric basis $\{b_1,\dots,b_n\}$ in terms of 
\begin{equation} \label{eq:u-def}
u(X,Y,s)
=\sup_{x\in B_X}u(x,Y,s)
=\sup_{x\in B_X} \Big\| \sum_{i=1}^{n}\min\{x_s^*,x_i^*\}b_i\Big\|_{Y},
\end{equation}
where $s=s(n,k)\in\IN$ is defined by
\begin{equation} \label{eq:s-def}
\frac{k}{\log(n/k+1)}<s\le 1+\frac{k}{\log(n/k+1)}.
\end{equation}
The characterization via $u(X,Y,s)$ has been applied, for instance, by Kaewtem~\cite{Kae17}, and Mayer and Ullrich \cite{MU21}, who proved results for entropy numbers of embeddings between mixed-norm spaces, but apparently has been largely overlooked.
For example, K\"uhn's lower bound \cite{Kue01} for $e_k(\id\colon \ell_p^n\to \ell_q^n)$ with $0<p<q\le \infty$   is a direct consequence (choose $x=s^{-1/p}\sum_{i=1}^{s}e_i$ in the supremum).  The quantity $u(X,Y,s)$ is related to the best $s$-term approximation in the worst case.
The latter concept is traditionally used in upper bounds for $\log n\le k\le n$ and will be discussed in Section~\ref{sec:sparse}.

We need the following formulation of \cite[Section 4, Theorem 2]{EN98}.

\begin{prop}\label{pro:edmunds-netrusov}
Let $n\in\N$. For $k<n/2$, we have
\[
e_k(\id\colon \ell_{p,u}^n\to \ell_{q,v}^n)
\asymp u(\ell_{p,u}^n,\ell_{q,v}^n,s),
\]
where $s\in\IN$ is as in \eqref{eq:s-def} and the implicit constants are independent of $k$ and $n$. 
\end{prop}

We shall deduce the following result; note that $s<n/\log(3)$ if $k<n/2$.
\begin{prop}\label{pro:truncation}
Let $0<p,q,u,v\le \infty$ and $s\in\N$ such that $1\le s<n/\log(3)$. Then we have the following asymptotics:
\begin{enumerate}
\item[\emph{(II)}] For $p<q=\infty$, we have 
\[
u(\ell_{p,u}^n,\ell_{\infty,v}^n,s)\asymp s^{-1/p}\log(s)^{1/v}.
\]
\item[\emph{(III.2)}] For $p=q<\infty$ and $u>v$, we have 
\[
u(\ell_{p,u}^n,\ell_{p,v}^n,s)\asymp \log(n/s+1)^{1/v-1/u}.
\]
\item[\emph{(IV.2)}] For $p=q=\infty$ and $u< v$, we have 
\[
u(\ell_{\infty,u}^n,\ell_{\infty,v}^n,s)\asymp \log(s+1)^{1/v-1/u}.
\]
\end{enumerate}
All implicit constants are independent of $s,n\in\IN$. 
\end{prop}

\begin{proof}
Let $0<p,u\le \infty$ and $x\in B_{p,u}^n$ be such that $x_1^*>\cdots>x_n^*>0$. It is easy to see that the supremum remains the same if we only consider such $x$'s. For $v<\infty$, we have 
\begin{equation}\label{eq:u}
\Big\| \sum_{i=1}^{n}\min\{x_s^*,x_i^*\}e_i\Big\|_{q,v}^v
=(x_s^*)^v\sum_{i=1}^{s}i^{v/q-1} +\sum_{i=s+1}^{n}(x_i^*)^v i^{v/q-1},
\end{equation}
whereas for $v=\infty$, we have
\[
\Big\| \sum_{i=1}^{n}\min\{x_s^*,x_i^*\}e_i\Big\|_{q,\infty}
=\max\{x_s^* s^{1/q}, \sup_{s+1\le i\le n}x_i^* i^{1/q}\}.
\]
We will use Lemma~\ref{lem:x-star-bound}, i.e., that
\begin{equation} \label{eq:x-star}
x_i^*\lesssim 
\begin{cases}
i^{-1/p}&\colon p<\infty,\\
(\log i)^{-1/u}&\colon p=\infty.
\end{cases}
\end{equation}

We distinguish several cases and carry out the computations only for $v<\infty$ (they are in fact easier for $v=\infty$).

\noindent \textbf{Case (II)} 

Let $p<q=\infty$. If $v<\infty$, we estimate \eqref{eq:u} by
\[
(x_s^*)^v\sum_{i=1}^{s}i^{-1} +\sum_{i=s+1}^{n}(x_i^*)^v i^{-1}
\lesssim s^{-v/p}\log s +\sum_{i=s+1}^{n}i^{-v/p-1}
\lesssim s^{-v/p}\log s.
\]
The upper bound for $v=\infty$ is a direct consequence of \eqref{eq:x-star}.

The lower bound is achieved by $x=s^{-1/p}\sum_{i=1}^{s}e_i$, which by \eqref{eq:fundamental-lorentz} satisfies
\[
\|x\|_{p,u}\asymp 1 \qquad \text{and}\qquad \|x\|_{\infty,v}\asymp s^{-1/p}(\log s)^{1/v}.
\]

\noindent \textbf{Case (III.2)} 

Let $0<p=q<\infty$ and $0<v<u\le \infty$. We first prove the upper bound.

By means of \eqref{eq:u} and \eqref{eq:x-star} we have
\begin{equation} \label{eq:III-2-up}
u(\ell_{p,u}^{n},\ell_{p,v}^n,s)^v
\lesssim 1+\sum_{i=s+1}^{n}(x_i^*)^v i^{v/p-1}.
\end{equation}
In the case of $u<\infty$, we use Hölder's inequality with $\beta=\frac{v}{p}-\frac{v}{u}$, $\varphi=u/v>1$ and $\varphi^*=u/(u-v)$, to obtain
\begin{align}
\notag\sum_{i=s+1}^{n}(x_i^*)^v i^{\beta} i^{-\beta}i^{v/p-1}
&\le
\Big(\sum_{i=s+1}^{n}(x_i^*)^{v\varphi}i^{\beta \varphi}\Big)^{1/\varphi}\cdot
\Big(\sum_{i=s+1}^{n}i^{(-\beta+v/p-1) \varphi^*}\Big)^{1/\varphi^*}\\
\label{eq:III.2}&= \Big(\sum_{i=s+1}^{n}(x_i^*)^{u}i^{u/p-1}\Big)^{v/u}
\cdot \Big(\sum_{i=s+1}^{n}i^{-1}\Big)^{(u-v)/u}\\
\notag&\lesssim \log(n/s)^{1-v/u}.
\end{align}
For $u=\infty$, we use $x_i^*\lesssim i^{-1/p}$ and obtain 
\[
\sum_{i=s+1}^{n}(x_i^*)^v i^{v/p-1}\lesssim \log(n/s).
\]
Combined with \eqref{eq:III-2-up}, this shows that 
\[
u(\ell_{p,u}^{n},\ell_{p,v}^n,s)\lesssim \log(n/s)^{1/v-1/u}.
\]
The lower bound is achieved by $x=s^{-1/p}\sum_{i=1}^{s}e_i+\sum_{i=s+1}^{n}i^{-1/p}e_i$, which satisfies 
\[
\|x\|_{p,u}
\asymp (1+ \log(n/s))^{1/u}\qquad \text{and}\qquad \|x\|_{p,v}\asymp (1+ \log(n/s))^{1/v}.
\]

\noindent \textbf{Case (IV.2)} 

Let $p=q=\infty$ and $0<u<v\le \infty$. We proceed as in the proof of case (II) and replace $s^{-1/p}$ by $(\log s)^{-1/u}$.

\end{proof}

We can now complete the proof of Theorem~\ref{thm:main} using Propositions~\ref{pro:edmunds-netrusov} and~\ref{pro:truncation}.

\begin{proof}[Proof of (II), (III.2) and (IV.2) for $k\le n$]

We combine Propositions~\ref{pro:edmunds-netrusov} and~\ref{pro:truncation} to obtain the asymptotics of 
\[
e_k(\id\colon \ell_{p,u}^n\to \ell_{q,v}^n)
\]
for $k<n/2$ in terms of $s\asymp\frac{k}{\log(n/k+1)}$. For $n/2\le k\le n$, we use monotonicity. All implicit constants are independent of $n$ and $k$. Further note that, for $k\le \log n$, we have
\[
s\le\frac{\log n}{\log(1+n/\log n)}\le C,
\]
where $C\in(0,\infty)$ is some absolute constant. Therefore, after looking at Proposition~\ref{pro:norm-emb}, monotonicity yields
\[
e_k(\id\colon \ell_{p,u}^n\to \ell_{q,v}^n)\asymp \|\id\colon \ell_{p,u}^n\to \ell_{q,v}^n\|.
\]

\noindent \textbf{Case (II)} ($p<\infty$ and $q=\infty$)

For $k<n/2$, we have
\[
e_k(\id\colon \ell_{p,u}^n\to \ell_{\infty,v}^n)
\asymp s^{-1/p}\log(s),
\]
which proves the theorem in this case.

\noindent \textbf{Case (III.2)} ($p=q<\infty$ and $u>v$) 

For $k<n/2$, we have 
\[
e_k(\id\colon \ell_{p,u}^n\to \ell_{\infty,v}^n)
\asymp \log(n/s+1)^{1/v-1/u}
\asymp \log(n/k+1)^{1/v-1/u},
\]
which proves the theorem in this case.

\noindent \textbf{Case (IV.2)} ($p=q=\infty$ and $u<v$) 

For $k<n/2$, we have 
\[
e_k(\id\colon \ell_{p,u}^n\to \ell_{\infty,v}^n)
\asymp \log(s+1)^{1/v-1/u},
\]
which proves the theorem in this case.
\end{proof}

\section{Sparse Approximation}\label{sec:sparse}

Since the proof of Theorem~\ref{thm:main} in the cases (II), (III.2), and (IV.2) in the intermediate range $\log n\le k\le n$, and in particular Proposition~\ref{pro:edmunds-netrusov} (taken from \cite[Section 4, Theorem 2]{EN98}), is very much related to ideas of sparse approximation, we provide some background and an alternative proof.

In general, for positive integers $s\le n$ the error of best $s$-term approximation of a vector $x\in\IR^n$ in a quasi-norm $\|\cdot\|_Y$ is given by
\[
\sigma_s(x)_Y=\inf\big\{\|x-z\|_Y\colon z\in\IR^n \text{ with }|\{i\colon z_i\neq 0\}|\le s\big\}.
\]
It measures (with respect to $\|\cdot\|_Y$) how far $x$ is from being $s$-sparse, i.e., how far from being supported on $s$ coordinates. In contrast, for obtaining the quantity $u(x,Y,s)$ in \eqref{eq:u-def} only truncation of the entries of $x$ is permitted. However, assuming symmetry, both quantities are suitable for characterizing the behavior of entropy numbers.

\begin{prop}\label{pro:entropy-sparse}
Let $n\in\IN$ and let $X,Y$ be $n$-dimensional quasi-Banach spaces with quasi-norm constants $C_X,C_Y\ge 1$ and a common symmetric basis $\{e_1,\dots,e_n\}$. For $k<n/2$, we have
\[
e_k(\id\colon X\to Y)
\asymp \sup_{x\in B_X}\sigma_s(x)_Y,
\]
where $s\in\IN$ is the minimal integer with $s>\frac{k}{\log(n/k+1)}$ and the implicit constants depends only on $\max\{C_X,C_Y\}$. Note that sparsity is with respect to the basis $\{e_1,\dots,e_n\}$.
\end{prop}
\begin{proof}
By \cite[Section 4, Theorem 2]{EN98} the statement holds with $\sup_{x\in B_X}\sigma_s(x)_Y$ replaced by $u(X,Y,s)$. We will show that in fact for $s<n/2$
\begin{equation} \label{eq:u-sigma}
\sup_{x\in B_X}\sigma_s(x)_Y\asymp u(X,Y,s),
\end{equation}
where the implicit constants are independent of $s$. Then it remains to note that for $k<n/2$ we also have $s<n/2$. Using symmetry and Lemma~\ref{lem:lattice}, the upper bound in \eqref{eq:u-sigma} follows from
\[
\sigma_s(x)_Y
=\left\|\sum_{i=s+1}^{n} x_{i}^* e_i \right\|_Y
\le K_Y\left\|\sum_{i=1}^{n} \min\{x_s^*,x_{i}^*\} e_i \right\|_Y
= K_Y u(x,Y,s),
\]
and taking the supremum over $x=\sum_{i=1}^{n}x_i e_i\in B_X$.

For the lower bound in \eqref{eq:u-sigma} we write
\begin{align*}
u(x,Y,2s)&=\left\|\sum_{i=1}^s x_{2s}^* e_i + \sum_{i=s+1}^{2s} x_{2s}^* e_i + \sum_{i=2s+1}^{n} x_{i}^* e_i\right\|_Y\\
&\le C_Y\left\{\left\|\sum_{i=1}^s x_{2s}^* e_i \right\|_Y + K_Y\left\| \sum_{i=s+1}^{n} x_{i}^* e_i \right\|_Y\right\}\le 2\,C_YK_Y \sigma_s(x)_Y,
\end{align*}
take the supremum, and note that by \cite[Section 4, Lemma 4 (i)]{EN98} we have, for $s<n/2$, 
\begin{equation} \label{eq:u-doubling}
u(X,Y,2s)\gtrsim u(X,Y,s)
\end{equation}
with an implicit constant only depending on $\max\{C_X,C_Y\}$.

\end{proof}

\begin{remark}\label{rem:temlyakov}
Proposition~\ref{pro:entropy-sparse} can be seen as a complement to Theorem 3.1 in \cite{Tem13} by Temlyakov who proves an upper bound on the entropy numbers under polynomial decay assumption on the best $s$-term approximation numbers uniformly over compact sets.
\end{remark}

We note the following consequence of \eqref{eq:u-sigma} and \eqref{eq:u-doubling}, which shows that best $s$-term approximation numbers exhibit regular decay.

\begin{cor}
Assume $X$ and $Y$ are as in Proposition~\ref{pro:entropy-sparse}. Then, for $s<n/2$, 
\[
\sup_{x\in B_X}\sigma_{2s}(x)_Y\asymp\sup_{x\in B_X}\sigma_s(x)_Y,
\]
where the implicit constants depend only on $\max\{C_X,C_Y\}$.
\end{cor}

In the following, we will give the above mentioned alternative proof of Theorem~\ref{thm:main}. Since $x_{s+1}^*=\sigma_s(x)_{\infty}$ holds, Lorentz quasi-norms can be understood via best $s$-term approximation. We believe the following estimates to be of independent interest. The case $u=\infty$ and $v=q>p$ is for example covered in \cite[Prop. 2.11]{FR13}. 

\begin{prop}\label{pro:s-term-app}
	Let $0<p,q,u,v\le \infty$ and $s\le n$ be positive integers and assume that $x\in\IR^n$. For $q=\infty$, we have
	\[
	\sigma_s(x)_{\infty,v}
	\lesssim
	\|x\|_{p,u}
	\begin{cases}
		s^{-1/p}(\log s)^{1/v}&\colon p<\infty,\\
		(\log s)^{1/v-1/u}&\colon p=\infty\text{ and } u<v,\\
	\end{cases}
	\]
	and, for $p=q<\infty$ and $v<u$, we have
	\[
	\sigma_s(x)_{p,v}
	\lesssim
	\|x\|_{p,u}
	(\log(n/s)+1)^{1/v-1/u}.
	\]
	All implicit constants are independent of $n$ and $s$.
\end{prop}

For the proof of Proposition~\ref{pro:s-term-app} we need the following.
\begin{lem}\label{lem:tail-asymp}
Let $s\le n$ be positive integers. Then the following estimates hold:
\begin{enumerate}[label=(\roman*)]
\item For $\lambda>0$, we have
\[
\sum_{i=s+1}^n (i-s)^{-1}i^{-\lambda}
\lesssim s^{-\lambda}\log s.
\]
\item For $\lambda>1$, we have
\[
\sum_{i=s+1}^n (i-s)^{-1}(\log i)^{-\lambda}
\lesssim (\log s)^{-\lambda+1}.
\]
\end{enumerate}
The implicit constants depend only on the parameter $\lambda$. 
\end{lem}
We postpone its proof and first use it to deduce Proposition~\ref{pro:s-term-app}.

\begin{proof}[Proof of Proposition~\ref{pro:s-term-app}]
We first prove the case $q=\infty$. If $v=\infty$, then we obtain from Lemma~\ref{lem:x-star-bound} that
\begin{equation} \label{eq:s-term}
\sigma_s(x)_{\infty}
=x_{s+1}^*
\lesssim
\begin{cases}
	s^{-1/p}&\colon p<\infty,\\	
	(\log s)^{-1/u}&\colon p=\infty.
\end{cases}
\end{equation}
If $v<\infty$, then 
\[
\sigma_s(x)_{\infty,v}^v
=\sum_{i=s+1}^{n}(i-s)^{-1}(x_i^*)^v .
\]
Combining \eqref{eq:s-term} with Lemma~\ref{lem:tail-asymp} (i) for $\lambda=v/p$ if $p<\infty$ and (ii) for $\lambda=v/u>1$ if $p=\infty$ completes the proof of the case $q=\infty$.

If $p=q<\infty$ and $v<u$, then 
\[
\sigma_s(x)_{p,v}^v
=\sum_{i=s+1}^{n}(i-s)^{v/p-1}(x_i^*)^v
\le \sum_{i=s+1}^{n}i^{v/p-1}(x_i^*)^v.
\]
The conclusion now follows by H\"older's inequality used as in the proof of Theorem~\ref{thm:main} (III.2), cf. \eqref{eq:III.2}.
\end{proof}

\begin{proof}[Proof of Lemma~\ref{lem:tail-asymp}]
We can assume that $n\ge 2s$, otherwise we increase $n$. We start with (i) and let $\lambda>0$. First, we decompose the sum as follows,
\begin{equation} \label{eq:decomposition}
\sum_{i=s+1}^n (i-s)^{-1}i^{-\lambda}
=\sum_{i=s+1}^{2s} (i-s)^{-1}i^{-\lambda}
+\sum_{i=2s+1}^n (i-s)^{-1}i^{-\lambda}.
\end{equation}
In the first sum on the right-hand side of \eqref{eq:decomposition}, due to monotonicity, we have $i^{-\lambda}\le s^{-\lambda}$. Therefore, 
\[
\sum_{i=s+1}^{2s} (i-s)^{-1}i^{-\lambda}
\le s^{-\lambda}\sum_{i=s+1}^{2s} (i-s)^{-1}
= s^{-\lambda}\sum_{i=1}^{s} i^{-1}
\lesssim s^{-\lambda} \log s.
\]
In the second sum on the right-hand side of \eqref{eq:decomposition}, we have $(i-s)^{-1}\le 2i^{-1}$. Thus, 
\[
\sum_{i=2s+1}^n (i-s)^{-1}i^{-\lambda}
\le 2 \sum_{i=2s+1}^n i^{-1-\lambda}
\lesssim s^{-\lambda}.
\]
Together, this completes the proof of (a).

For the proof of (b) let $\lambda>1$. We can decompose similarly to \eqref{eq:decomposition} and due to monotonicity of $(\log i)^{-\lambda}$ the bound on the first sum is analogous. In order to bound the second sum we note that
\[
\sum_{i=2s+1}^n (i-s)^{-1}(\log i)^{-\lambda}
\le 2 \sum_{i=2s+1}^n i^{-1}(\log i)^{-\lambda}
\lesssim s^{-1}(\log s)^{-\lambda+1}.
\]
This proves (b). 
\end{proof}

Combined, Propositions~\ref{pro:entropy-sparse} and~\ref{pro:s-term-app} can be used to replace Propositions~\ref{pro:edmunds-netrusov} and~\ref{pro:truncation} in the proof of Theorem~\ref{thm:main} in the cases (II), (III.2) and (IV.2) for the upper bounds.

The lower bounds in the proof of Theorem~\ref{thm:main} in the cases (II), (III.2) and (IV.2) can be proven via the following combinatorial lemma, which has been used for bounds on entropy numbers, in coding theory and compressed sensing (see, e.g., \cite[Lemma 9]{DV20} and the references given there). 

\begin{lem}\label{lem:combinatorial}
	Let $s\le n$ be positive integers. There are $T_1,\dots,T_M\subset \{1,\dots,n\}$ with
\begin{enumerate}[label=(\roman*)]
	\item $M\ge (n/4s)^{s/2}$,
	\item $|T_i|=s$ for $i=1,\dots,M$,
	\item $|T_i\cap T_j|<s/2$ for $i\neq j$.
\end{enumerate}
\end{lem}

In the cases (II) and (IV.2) we can use indicators $\mathbf{1}_{T_1},\dots,\mathbf{1}_{T_M}$ based on the sets $T_1,\dots,T_M$ in Lemma~\ref{lem:combinatorial} with $s=\ell(n,k)$ as in Theorem~\ref{thm:main}. Renormalizing these indicators gives us a large set of well-separated unit vectors and thus a lower bound on the entropy numbers for $\log n\le k\le n$ (see Step 4 in the proof of \cite[Theorem 2]{KV20}). 

In the case (III.2) we can use appropriately rescaled indicators of different sizes, adapting the arguments used in the proof of \cite[Theorem 10]{DV20} which are similar to the more elaborate approach used in the proof of \cite[Section 4, Theorem 2]{EN98}. For convenience of the reader we sketch the argument. First, let us note that if $E_1,\dots,E_n$ are disjoint subsets of $\IN$ with cardinality $\# E_\ell\asymp 4^{\ell}$ and $\alpha_1,\dots,\alpha_n\in\IR$, then, for $0<p,u\le \infty$, we have
\begin{equation} \label{eq:dyadic}
\Big\| \sum_{i=1}^{n}\alpha_i \mathbf{1}_{E_\ell}\Big\|_{p,u}
\asymp 
\begin{cases}
\Big(\sum_{\ell=1}^{n}4^{\ell u/p}|\alpha_\ell|^{u}\Big)^{1/u}&\colon u<\infty,\\
\max\limits_{1\le \ell\le n}4^{\ell/p}|\alpha_{\ell}|&\colon u=\infty,
\end{cases}
\end{equation}
where for $p=\infty$ we use $a/\infty=0$ for any $a\in\IR$ (see \cite[Lemma 6]{EN11}).

Following the proof of \cite[Theorem 10]{DV20}, let $n\in\N$ be sufficiently large and $\nu\ge 1$ be the largest integer such that $12\cdot 4^{\nu}\le n$ and $\mu$ be the smallest integer such that $k\le 4^{\mu}/2$. We obtain from Lemma~\ref{lem:combinatorial} that, for $n\in\N$ sufficiently large, $M\ge (n/4^{\mu+1})^{4^{\mu}/2}$ families $\{\widetilde{T}_j^{\ell}\colon \mu\le \ell\le \nu\}$, $1\le \mu\le \nu$, of such sets such that
\[
\frac{2}{3}4^{\ell}\le |\widetilde{T}_j^{\ell}|\le 4^{\ell},\quad \mu\le \ell \le \nu,\, 1\le j\le M,
\]
the sets $\widetilde{T}_j^{\ell}, \mu\le \ell\le \nu$ are mutually disjoint, and 
\[
|\widetilde{T}_i^{\ell}\cap \widetilde{T}_j^{\ell}|\le \frac{1}{2}4^{\ell},\quad \mu\le \ell \le \nu,\, i\neq j.
\]
Defining the vectors 
\[
x^j := \sum_{\ell=\mu}^{\nu}4^{-\ell/p}\mathbf{1}_{\widetilde{T}_j^{\ell}},\quad j=1,\dots,M,
\]
we obtain from \eqref{eq:dyadic} that
\[
\|x^j\|_{p,u}\asymp (\nu-\mu+1)^{1/u}
\]
and 
\[
\|x^i-x^j\|_{p,v}\gtrsim (\nu-\mu+1)^{1/v},\quad i\neq j.
\]
By rescaling, we can ensure that the points $x^j$ are in $B_{p,u}^n$ and are pairwise separated in the quasi-norm of $\ell_{p,v}^n$ by $\gtrsim (\nu-\mu+1)^{1/v-1/u}$. Noting that $\nu-\mu+1\gtrsim \log(n/k+1)$ the proof can be concluded as in Step 4 of the proof of Theorem 10 in \cite{DV20}. 

\subsection*{Acknowledgement}
We would like to thank G.~Schechtman for providing us with the proof of Lemma~\ref{lem:lattice} and the anonymous referee for the valuable comments, which helped to improve the manuscript.
Joscha Prochno’s research is supported by the German Research Foundation (DFG) under project 516672205 and by the Austrian Science Fund (FWF) under project P-32405.
This research was funded in whole or in part by the Austrian Science Fund (FWF) [Grant DOI: 10.55776/P32405; 10.55776/J4777].
The work of Jan Vyb\'\i ral has been supported by the grant P202/23/04720S of the Grant Agency of the Czech Republic.
For open access purposes, the authors have applied a CC BY public copyright license to any author-accepted manuscript version arising from this submission.

\bibliographystyle{plain}
\bibliography{entropy}

\end{document}